\documentclass[12pt,a4paper]{article}
\usepackage{}
\usepackage{amsfonts}
\usepackage{amssymb}
\usepackage{mathrsfs}
\usepackage{amsmath}
\usepackage{amsthm}
\usepackage{enumerate}
\usepackage{cite}
\usepackage[all]{xy}
\allowdisplaybreaks
\newtheorem{defn}{Definition}[section]
\newtheorem{thm}[defn]{Theorem}
\newtheorem{lem}[defn]{Lemma}
\newtheorem{prop}[defn]{Proposition}
\newtheorem{cor}[defn]{Corollary}
\newtheorem{eg}[defn]{Example}
\newtheorem{re}[defn]{Remark}

\newcommand\relphantom[1]{\mathrel{\phantom{#1}}}

\newcommand{\bdefn}{\begin{defn}}
\newcommand{\edefn}{\end{defn}}
\newcommand{\bthm}{\begin{thm}}
\newcommand{\ethm}{\end{thm}}
\newcommand{\blem}{\begin{lem}}
\newcommand{\elem}{\end{lem}}
\newcommand{\bprop}{\begin{prop}}
\newcommand{\eprop}{\end{prop}}
\newcommand{\bcor}{\begin{cor}}
\newcommand{\ecor}{\end{cor}}
\newcommand{\beg}{\begin{eg}}
\newcommand{\eeg}{\end{eg}}
\newcommand{\bre}{\begin{re}}
\newcommand{\ere}{\end{re}}
\newcommand{\bpf}{\begin{proof}}
\newcommand{\epf}{\end{proof}}

\newcommand{\Z}{{\rm Z}}
\newcommand{\C}{{\rm C}}

\newcommand{\BDer}{{\rm BDer}}

\newcommand{\R}{\mathcal{R}}

\newcommand{\benu}{\begin{enumerate}}
\newcommand{\eenu}{\end{enumerate}}
\newcommand{\bc}{\begin{center}}
\newcommand{\ec}{\end{center}}
\newcommand{\bea}{\begin{eqnarray}}
\newcommand{\eea}{\end{eqnarray}}
\newcommand{\Bea}{\begin{eqnarray*}}
\newcommand{\Eea}{\end{eqnarray*}}
\newcommand{\beq}{\begin{equation}}
\newcommand{\eeq}{\end{equation}}
\newcommand{\Beq}{\begin{equation*}}
\newcommand{\Eeq}{\end{equation*}}
\newcommand{\bspl}{\begin{split}}
\newcommand{\espl}{\end{split}}

\numberwithin{equation}{section}

\begin{document}
\title{{\bf Conformal biderivations of loop $W(a,b)$ Lie conformal algebra}}
\author{ Jun Zhao$^{1}$, Liangyun Chen$^{1*}$, Lamei Yuan$^{2}$
 \date{{\small $^{1}$School of Mathematics and Statistics, Northeast Normal
 University,\\
Changchun 130024, China\\
 $^{2}$Department of Mathematics, Harbin Institute of Technology, Harbin 150001, China}}}

\maketitle
\date{}

\begin{abstract}
In this paper, we study conformal biderivations of a Lie conformal algebra. First, we give the definition of conformal biderivation. Next, we determine the conformal biderivations of loop $W(a,b)$ Lie conformal algebra,
 loop Virasoro Lie conformal algebra and Virasoro Lie conformal algebra. Especially, all conformal biderivations on Virasoro Lie conformal algebra are inner conformal biderivations.

\bigskip

\end{abstract}
\renewcommand{\thefootnote}{\fnsymbol{footnote}}
\footnote[0]{ Corresponding author(L. Chen): chenly640@nenu.edu.cn.}
\footnote[0]{Supported by NNSF of China (Nos. 11771069 and 11301109), NNSF of Jilin province (No. 20170101048JC) and the project of jilin province department of education (No. JJKH20180005K).}

\section{Introduction}
The notion of Lie conformal algebras was introduced by V. G. Kac as a formal language describing the singular part of the operator product expansion in conformal field theory. It is useful to research infinite dimensional Lie algebras satisfying the locality property. The structure theory and representation theory of some Lie conformal algebras have been extensively studied in \cite{BKV,DK}.

In recent years, biderivations have been aroused many scholar's great interests. In \cite{B1,B2}, Bre$\check{s}$ar et al. showed that all biderivations on commutative prime rings are inner biderivations and they determined the biderivations of semiprime rings.  In \cite{BZ}, Zhao Kaiming proved all skew-symmetric biderivations on a perfect and centerless Lie algebra are inner biderivations. In \cite{CSY,HWX,LGZ,T1,T2,WY}, authors give biderivations of specific examples of Lie algebras.

The main object investigated in this paper is the loop $W(a,b)$ Lie conformal algebra, denoted by $\rm CLW(a,b)$, which is a free $\mathbb{C}[\partial]$-module $\rm{CLW(a,b)}=(\bigoplus_{i\in\mathbb{Z}}\mathbb{C}[\partial]L_i)\oplus(\bigoplus_{i\in\mathbb{Z}}\mathbb{C}[\partial]G_i)$   with a $\mathbb{C}[\partial]$-basis $\{L_i,G_i|i\in\mathbb{Z}\}$ satisfying the following $\lambda$-brackets:
\begin{eqnarray*}
&&[{L_i}_\lambda {L_j}]=(\partial+2\lambda)L_{i+j},\, [{L_i}_\lambda {G_j}]=(\partial+(1-b)\lambda)G_{i+j},
\\ &&[{G_i}_\lambda {L_j}]=-(b\partial+(b-1)\lambda)G_{i+j},\, [{G_i}_\lambda {G_j}]=0.
\end{eqnarray*}
The relations between $W(a,b)$ and $\rm CLW(a,b)$ can be found in \cite{FWY}.

This paper is organized as follows. First, we give the definitions of conformal biderivation and inner conformal biderivation on a Lie conformal algebra. Next, we determine the conformal biderivations of loop $W(a,b)$ Lie conformal algebra, loop Virasoro Lie conformal algebra and Virasoro Lie conformal algebra. Especially, all conformal biderivations on Virasoro Lie conformal algebra are inner conformal biderivations.

Throughout this paper, all vector
spaces, linear maps, and tensor products are over the complex field $\mathbb{C}$. In addition to the standard
notations $\mathbb{Z}$ and $\mathbb{R}$, we use $\mathbb{Z}_{\geq 0}$ to denote the set of nonnegative integers.

\section{Conformal biderivations of a Lie conformal algebra}

The following notion was due to \cite{BKV}.
\bdefn\rm
A Lie conformal algebra is a $\mathbb{C}[\partial]$-module $\R$ endowed
with a bilinear map $\R\times\R\rightarrow\R[\lambda], (a,b)\mapsto[a_\lambda b]$, called the $\lambda$-bracket, satisfying the
following axioms ($x,y,z\in\R$):
\begin{enumerate}
\item[$(1)$] {\rm  Conformal\ sesquilinearity}:$ [(\partial x)_\lambda y]=-\lambda[x_\lambda y];$
\item[$(2)$] {\rm Skew-symmetry}:  $[x_\lambda y]=-[y_{-\partial-\lambda}x];$
\item[$(3)$] {\rm Jacobi\ identity}: $[x_\lambda[y_\mu z]]=[[x_\lambda y]_{\lambda+\mu}z]+[y_\mu[x_\lambda z]].$
\end{enumerate}
\edefn

\beg\rm\cite{DK}{\it Virasoro Lie conformal algebra} $\rm Vir$ is a free $\mathbb{C}[\partial]$-module $\rm{Vir}=\mathbb{C}[\partial]L$, satisfying the following $\lambda$-bracket:
\begin{eqnarray*}
[L_\lambda L]=(\partial+2\lambda)L.
\end{eqnarray*}
\eeg

\beg\rm\cite{WCY} {\it loop Virasoro Lie conformal algebra} $\rm CW$ is a free $\mathbb{C}[\partial]$-module
$\rm{CW}=\bigoplus_{i\in\mathbb{Z}}\mathbb{C}[\partial]L_i$ satisfying the following $\lambda$-brackets:
\begin{eqnarray*}
&&[{L_i}_\lambda {L_j}]=(\partial+2\lambda)L_{i+j}.
\end{eqnarray*}
\eeg

\beg\rm\cite{FWY} {\it loop $W(a,b)$ Lie conformal algebra $\rm CLW(a,b)$} is a free $\mathbb{C}[\partial]$-module $\rm{CLW(a,b)}=(\bigoplus_{i\in\mathbb{Z}}\mathbb{C}[\partial]L_i)\oplus(\bigoplus_{i\in\mathbb{Z}}\mathbb{C}[\partial]G_i)$, satisfying the following $\lambda$-brackets:
\begin{eqnarray*}
&&[{L_i}_\lambda {L_j}]=(\partial+2\lambda)L_{i+j},\, [{L_i}_\lambda {G_j}]=(\partial+(1-b)\lambda)G_{i+j},
\\ &&[{G_i}_\lambda {L_j}]=-(b\partial+(b-1)\lambda)G_{i+j},\, [{G_i}_\lambda {G_j}]=0.
\end{eqnarray*}
\eeg

\bdefn\rm\cite{DK}
Let $U$ and $V$ be two $\C[\partial]$-modules. A {\it conformal linear map} from $U$ to $V$ is a
 $\mathbb{C}$-linear map $f:U\rightarrow V[\lambda],$ denoted by $f_\lambda:U\rightarrow V[\lambda]$, such that
 $f_\lambda(\partial x)=(\partial+\lambda)f_\lambda(x)$.

Moreover, Let $W$ also be a $\C[\partial]$-module. A {\it conformal bilinear map} from $U\times V$ to $W$ is a
$\mathbb{C}$-bilinear map $f:U\times V\rightarrow W[\lambda]$, denoted by $f_\lambda:U\times V\rightarrow
W[\lambda]$, such that $f_\lambda(\partial x,y)=-\lambda f_\lambda(x,y)$ and $f_\lambda(x,\partial
y)=(\partial+\lambda)f_\lambda(x,y)$.
\edefn

\bdefn\rm Let $\R$ be a Lie conformal algebra. We call a conformal bilinear map
$\phi_\lambda:\R\times\R\rightarrow\R[\lambda]$ {\it skew-symmetric} if it satisfies
$\phi_\lambda(x,y)=-\phi_{-\partial-\lambda}(y,x)$.                                                                                                   \edefn

\bdefn\rm Let $\R$ be a Lie conformal algebra. We call a conformal bilinear map $\phi_\lambda:\R\times\R\rightarrow\R[\lambda]$ a {\it conformal biderivation} of $\R$ if it satisfies the following equations:
\begin{eqnarray}
&&\phi_\lambda(x,y)=-\phi_{-\partial-\lambda}(y,x);\label{def1a}\\
&&\phi_{\lambda}(x,[y_\mu z])=[(\phi_\lambda(x,y))_{\lambda+\mu}z]+[y_\mu\phi_\lambda(x,z)]\label{def1b}.
\end{eqnarray}
\edefn

\begin{re}\rm
If $\phi_\lambda$ is a conformal biderivation of Lie conformal algebra $\R$, then $\phi_\lambda(x,\cdot)$ is a conformal derivation obviously.
\end{re}

\begin{lem}
Let $\phi_\lambda$ be a conformal biderivation of Lie conformal algebra $\R$. Then Eq.\eqref{def1b} is equivalent to
\begin{eqnarray}
\phi_{\lambda+\mu}([x_\mu y],z)=[x_\mu\phi_\lambda(y,z)]-[y_\lambda\phi_\mu(x,z)]\label{lem1}.
\end{eqnarray}
\end{lem}
\begin{proof}
If Eq.\eqref{lem1} satisfies, by conformal sesquilinearity, we have
\begin{eqnarray*}
-\phi_{-\partial-\lambda-\mu}(z,[x_\mu y])=[x_\mu\phi_\lambda(y,z)]-[y_\lambda\phi_\mu(x,z)].
\end{eqnarray*}
Replace $z,x,y$ by $x,y,z$, we obtain
\begin{eqnarray*}
-\phi_{-\partial-\lambda-\mu}(x,[y_\mu z])=[y_\mu\phi_\lambda(z,x)]-[z_\lambda\phi_\mu(y,x)].
\end{eqnarray*}
By conformal sesquilinearity, we get
\begin{eqnarray*}
-\phi_{-\partial-\lambda-\mu}(x,[y_\mu z])=-[y_\mu\phi_{-\partial-\lambda}(x,z)]+
[z_\lambda\phi_{-\partial-\mu}(x,y)].
\end{eqnarray*}
Replace $\lambda, \mu$ by $-\partial-\lambda-\mu, \mu$ respectively and by conformal sesquilinearity,
\begin{eqnarray*}
\phi_{\lambda}(x,[y_\mu z])=[(\phi_\lambda(x,y))_{\lambda+\mu}z]+[y_\mu\phi_\lambda(x,z)]
\end{eqnarray*}
satisfies.  The reverse conclusion follows similarly.
\end{proof}

\bdefn\rm
Denote by $\BDer(\R)$ the set of all conformal biderivations of $\R$.
\edefn

\begin{lem}
If the map $\phi^t_\lambda:\R\times\R\rightarrow\R[\lambda]$, defined by $\phi^t_\lambda(x,y)=t[x_\lambda y]$ for all $x,y\in\R$, where $t\in\mathbb{C}$, then $\phi^t$ is a conformal biderivation of $\R$. We call this class conformal biderivations by {\it inner conformal biderivations}.
\end{lem}
\begin{proof}
It is straightforward by the definition of conformal biderivation.
\end{proof}

\begin{lem}
Let $\phi_\lambda$ be a conformal biderivation of Lie conformal algebra $\R$. Then
\begin{eqnarray}
[(\phi_{\mu}(x, y))_{\mu+\gamma}[u_\lambda v]]=[[x_\mu y]_{\mu+\gamma}\phi_\lambda(u,v)]\label{lem2}
\end{eqnarray}
for any $x,y,u,v\in\R$.
\end{lem}
\begin{proof}
On the one hand, using Eq.\eqref{lem1}, we have
\begin{eqnarray*}
\phi_{\lambda+\mu}([x_\mu u],[y_\gamma v])&=&[x_\mu\phi_\lambda(u,[y_\gamma v])]-[u_\lambda\phi_\mu(x,[y_\gamma v])]\\
&=&([x_\mu[(\phi_\lambda(u,y))_{\lambda+\gamma}v]]+[x_\mu[y_\gamma\phi_\lambda(u,v)]])\\
&&-([u_\lambda[(\phi_\mu(x,y))_{\mu+\gamma}v]]+[u_\lambda[y_\gamma\phi_\mu(x,v)]]).
\end{eqnarray*}

On the other hand, using Eq.\eqref{def1b}, we have
\begin{eqnarray*}
\phi_{\lambda+\mu}([x_\mu u],[y_\gamma v])
&=&[(\phi_{\lambda+\mu}([x_\mu u],y))_{\lambda+\mu+\gamma}v]+[y_\gamma\phi_{\lambda+\mu}([x_\mu u],v)]\\
&=&([[x_\mu\phi_\lambda(u,y)]_{\lambda+\mu+\gamma}v]
-[[u_\lambda\phi_\mu(x,y)]_{\lambda+\mu+\gamma}v])\\
&&+([y_\gamma[x_\mu\phi_\lambda(u,v)]]
-[y_\gamma[u_\lambda\phi_\mu(x,v)]]).
\end{eqnarray*}
Comparing two sides of the above equations, and using the Jacobi identity of Lie conformal algebra, we obtain
\begin{eqnarray*}
[(\phi_\lambda(u,y))_{\lambda+\gamma}[x_\mu v]]+[[x_\mu y]_{\mu+\gamma}\phi_\lambda(u,v)]
=[(\phi_\mu(x,y))_{\mu+\gamma}[u_\lambda v]]+[[u_\lambda y]_{\lambda+\gamma}\phi_\mu(x,v)],
\end{eqnarray*}
it is equivalent to
\begin{eqnarray*}
[(\phi_\mu(x,y))_{\mu+\gamma}[u_\lambda v]]-[[x_\mu y]_{\mu+\gamma}\phi_\lambda(u,v)]
=[(\phi_\lambda(u,y))_{\lambda+\gamma}[x_\mu v]]-[[u_\lambda y]_{\lambda+\gamma}\phi_\mu(x,v)].
\end{eqnarray*}

Now let $\Phi_{\lambda,\mu,\gamma}(x,y;u,v)=[(\phi_\mu(x,y))_{\mu+\gamma}[u_\lambda v]]-[[x_\mu y]_{\mu+\gamma}\phi_\lambda(u,v)]$.
From the above equation, it follows at once that
\begin{eqnarray*}
\Phi_{\lambda,\mu,\gamma}(x,y;u,v)=\Phi_{\mu,\lambda,\gamma}(u,y;x,v).
\end{eqnarray*}
Obviously, by skew-symmetry and conformal sesquilinearity , we get
\begin{eqnarray*}
&&\Phi_{\lambda,\mu,\gamma}(x,y;u,v)\\
&=&[(\phi_\mu(x,y))_{\mu+\gamma}[u_\lambda v]]-[[x_\mu y]_{\mu+\gamma}\phi_\lambda(u,v)]\\
&=&-[(\phi_{-\partial-\mu}(y,x))_{\mu+\gamma}[u_\lambda v]]+[[y_{-\partial-\mu}x]_{\mu+\gamma}\phi_\lambda(u,v)]\\
&=&-[(\phi_{\gamma}(y,x))_{\mu+\gamma}[u_\lambda v]]+[[y_{\gamma}x]_{\mu+\gamma}\phi_\lambda(u,v)]\\
&=&-\Phi_{\lambda,\gamma,\mu}(y,x;u,v).
\end{eqnarray*}

For one thing, we have
\begin{eqnarray*}
\Phi_{\lambda,\mu,\gamma}(x,y;u,v)=-\Phi_{\lambda,\gamma,\mu}(y,x;u,v)
=-\Phi_{\gamma,\lambda,\mu}(u,x;y,v)
=\Phi_{\gamma,\mu,\lambda}(x,u;y,v).
\end{eqnarray*}

For another, we also have
\begin{eqnarray*}
\Phi_{\lambda,\mu,\gamma}(x,y;u,v)
&=&\Phi_{\mu,\lambda,\gamma}(u,y;x,v)
=-\Phi_{\mu,\gamma,\lambda}(y,u;x,v)\\
&=&-\Phi_{\gamma,\mu,\lambda}(x,u;y,v)
=-\Phi_{\gamma,\mu,\lambda}(x,u;y,v).
\end{eqnarray*}
Then
\begin{eqnarray*}
\Phi_{\lambda,\mu,\gamma}(x,y;u,v)=0,
\end{eqnarray*}
that is, $[(\phi_{\mu}(x, y))_{\mu+\gamma}[u_\lambda v])=[[x_\mu y]_{\mu+\gamma}\phi_\lambda(u,v)].$
\end{proof}

\begin{re}\rm \label{re2.12}
Let $\phi_\lambda$ be a conformal biderivation of Lie conformal algebra $\R$. If $[x_\lambda y]=0,$ then $\phi_\lambda(x,y)\in\Z[\R_\lambda\R]$.
\end{re}

\section{Conformal biderivations of Lie conformal algebras $\rm CW$, $\rm Vir$ and $\rm CLW(a,b)$}

\begin{thm}\label{thm1}
Every conformal biderivation $\phi_\lambda$ on the loop Virasoro Lie conformal algebra $CW$ has the following forms:
\begin{eqnarray*}
\phi_\lambda(L_i,L_j)=(\partial+2\lambda)\sum_{k\in\mathbb{Z}}a_{k-i-j}L_k,
\end{eqnarray*}
for some complex numbers set $\{a_k,k\in\mathbb{Z}\}$.
\end{thm}
\begin{proof}
Suppose that $\phi_\lambda(L_i,L_j)=\sum_{k\in\mathbb{Z}}f^k_{i,j}(\partial,\lambda)L_k$ is a conformal biderivation of the loop Virasoro Lie conformal algebra $\rm CW$.

Take $x=L_i, y=L_j, u=L_i, v=L_j$ in Eq.\eqref{lem2}, we get
\begin{eqnarray*}
[(\phi_{\mu}(L_i,L_j))_{\mu+\gamma}[{L_i}_\lambda L_j]]=[[{L_i}_\mu L_j]_{\mu+\gamma}\phi_\lambda(L_i,L_j)],
\end{eqnarray*}
that is
\begin{eqnarray*}
(\partial+\mu+\gamma+2\lambda)\sum_{k\in\mathbb{Z}}f^k_{i,j}(-\mu-\gamma,\mu)L_{i+j+k}
=(\mu-\gamma)\sum_{k\in\mathbb{Z}}f^k_{i,j}(\partial+\mu+\gamma,\lambda)L_{i+j+k},
\end{eqnarray*}
obviously, we have
\begin{eqnarray*}
(\partial+\mu+\gamma+2\lambda)f^k_{i,j}(-\mu-\gamma,\mu)
=(\mu-\gamma)f^k_{i,j}(\partial+\mu+\gamma,\lambda),
\end{eqnarray*}
for any $i,j,k\in\mathbb{Z}$.
Fix $i,j,k\in\mathbb{Z}$, considering the power of $\partial$, we have $f^k_{i,j}(\partial,\lambda)=a_0(\lambda)+a_1(\lambda)\partial$, substituting into the above equation, we get
\begin{eqnarray*}
(\partial+\mu+\gamma+2\lambda)(a_0(\mu)+a_1(\mu)(-\mu-\gamma))
=(\mu-\gamma)(a_0(\lambda)+a_1(\lambda)(\partial+\mu+\gamma)).
\end{eqnarray*}
Considering the coefficients of $\partial$, we have
\begin{eqnarray*}
a_0(\mu)+a_1(\mu)(-\mu-\gamma)=a_1(\lambda)(\mu-\gamma),
\end{eqnarray*}
it is easy to get $a_1(\lambda)=a_1\in\mathbb{C}, a_0(\lambda)=2\lambda a_1$, so $f^k_{i,j}(\partial,\lambda)=a_1(\partial+2\lambda).$ So $\phi_\lambda(L_i,L_j)=(\partial+2\lambda)\sum_{k\in\mathbb{Z}}f^k_{i,j}L_k$, where $f^k_{i,j}\in\mathbb{C}$.

Take $x=L_i, y=L_j, u=L_k, v=L_l$ in Eq.\eqref{lem2} again, we can get
$f^k_{i,j}=f^{k+m+n-i-j}_{m,n},$ especially, $f^k_{0,0}=f^{k+m+n}_{m,n}$.
Set $f^k_{0,0}:=a_k$, i.e. $\phi_\mu(L_0,L_0)=(\partial+2\lambda)\sum_{k\in\mathbb{Z}}a_kL_k$. Then
\begin{eqnarray}
&&\phi_\lambda(L_i,L_j)=(\partial+2\lambda)\sum_{k\in\mathbb{Z}}a_{k-i-j}L_k.\label{lem3a}
\end{eqnarray}

Conversely, it is easy to prove that any skew-symmetric conformal bilinear map satisfying Eq.\eqref{lem3a} is a conformal biderivation.
\end{proof}

\begin{cor}
Every conformal biderivation on the Virasoro Lie conformal algebra is an inner conformal biderivation.
\end{cor}
\begin{proof}
By the proof of Theorem \ref{thm1}, we can get the conclusion immediately.
\end{proof}

\begin{thm}
Let $\R$ be the Lie conformal algebra $\rm CLW(a,b)$. Then a skew-symmetric conformal bilinear map $\phi_\lambda$ of $\R$ is a conformal biderivation if and only if $\phi_\lambda$ satisfies the following conditions:
\begin{eqnarray*}
&&\phi_\lambda(L_i,L_j)=(\partial+2\lambda)(\sum_{k\in\mathbb{Z}}a_{k-i-j}L_k
+\delta_{b+1,0}\sum_{k\in\mathbb{Z}}b_{k-i-j}G_k),\\
&&\phi_\lambda(G_i,G_j)=0,\,\,\,
\phi_\lambda(L_i,G_j)=(\partial+(1-b)\lambda)\sum_{k\in\mathbb{Z}}a_{k-i-j}G_k,
\end{eqnarray*}
for some two complex numbers sets $\{a_k, k\in\mathbb{Z}\}$ and $\{b_k, k\in\mathbb{Z}\}$.
\end{thm}
\begin{proof}
Let $\phi_\lambda$ be a conformal biderivation of Lie conformal algebra $\R$. By Remark \eqref{re2.12}, $\phi_\lambda(G_i,G_j)=0$ obviously. Suppose that
\begin{eqnarray*}
&&\phi_\lambda(L_i,L_j)=\sum_{k\in\mathbb{Z}}f^k_{i,j}(\partial,\lambda)L_k
+\sum_{k\in\mathbb{Z}}g^k_{i,j}(\partial,\lambda)G_k,\\
&&\phi_\lambda(L_i,G_j)=\sum_{k\in\mathbb{Z}}s^k_{i,j}(\partial,\lambda)L_k
+\sum_{k\in\mathbb{Z}}d^k_{i,j}(\partial,\lambda)G_k.
\end{eqnarray*}

Firstly, take $x=L_i, y=L_j, u=L_i, v=G_j$ in Eq.\eqref{lem2}, we get
\begin{eqnarray*}
&&[(\phi_{\mu}(L_i, L_j))_{\mu+\gamma}[{L_i}_\lambda G_j]]=[[{L_i}_\mu L_j]_{\mu+\gamma}\phi_\lambda(L_i,G_j)]\\
\Longleftrightarrow&&[(\sum_{k\in\mathbb{Z}}f^k_{i,j}(\partial,\mu)L_k
+\sum_{k\in\mathbb{Z}}g^k_{i,j}(\partial,\mu)G_k)_{\mu+\gamma}(\partial+(1-b)\lambda)G_{i+j}]\\&&=
[((\partial+2\mu){L_{i+j}})_{\mu+\gamma}(\sum_{k\in\mathbb{Z}}s^k_{i,j}(\partial,\lambda)L_k
+\sum_{k\in\mathbb{Z}}d^k_{i,j}(\partial,\lambda)G_k)]\\
\Longleftrightarrow&&(\partial+\mu+\gamma+(1-b)\lambda)(\partial+(1-b)(\mu+\gamma))
\sum_{k\in\mathbb{Z}}f^k_{i,j}(-\mu-\gamma,\mu)G_{i+j+k}\\&&=
(\mu-\gamma)(\partial+2\mu+2\gamma)\sum_{k\in\mathbb{Z}}s^k_{i,j}(\partial+\mu+\gamma,\lambda)L_{i+j+k}\\
&&+(\mu-\gamma)(\partial+(1-b)(\mu+\gamma))\sum_{k\in\mathbb{Z}}d^k_{i,j}(\partial+\mu+\gamma,\lambda)G_{i+j+k},
\end{eqnarray*}
obviously, we have $s^k_{i,j}(\partial,\lambda)=0$ and
\begin{eqnarray*}
&&(\partial+\mu+\gamma+(1-b)\lambda)f^k_{i,j}(-\mu-\gamma,\mu)=(\mu-\gamma)d^k_{i,j}(\partial+\mu+\gamma,\lambda).
\end{eqnarray*}
for any $i,j,k\in\mathbb{Z}$. Considering the coefficients of $\partial$, we have $d^k_{i,j}(\partial,\lambda)=a_0(\lambda)+a_1(\lambda)\partial$, substituting into the above equation, we get
\begin{eqnarray*}
&&(\partial+\mu+\gamma+(1-b)\lambda)f^k_{i,j}(-\mu-\gamma,\mu)
=(\mu-\gamma)(a_0(\lambda)+a_1(\lambda)(\partial+\mu+\gamma)).
\end{eqnarray*}
Considering the coefficients of $\gamma$ again, we get $f^k_{i,j}(\partial,\lambda)=b_0(\lambda)+b_1(\lambda)\partial$, substituting into the above equation, we get
\begin{eqnarray*}
&&(\partial+\mu+\gamma+(1-b)\lambda)(b_0(\mu)+b_1(\mu)(-\mu-\gamma))
=(\mu-\gamma)(a_0(\lambda)+a_1(\lambda)(\partial+\mu+\gamma)),
\end{eqnarray*}
observing the coefficients of $\gamma^2$, we have $b_1(\mu)=a_1(\lambda)=a_1\in\mathbb{C}$, substituting into the above equation, we get
\begin{eqnarray*}
b_0(\mu)=2 a_1\mu, a_0(\mu)=a_1(1-b)\mu.
\end{eqnarray*}
So we can set
\begin{eqnarray*}
&&\phi_\lambda(L_i,L_j)=(\partial+2\lambda)\sum_{k\in\mathbb{Z}}d^k_{i,j}L_k
+\sum_{k\in\mathbb{Z}}g^k_{i,j}(\partial,\lambda)G_k,\\
&&\phi_\lambda(G_i,G_j)=0,\\
&&\phi_\lambda(L_i,G_j)=(\partial+(1-b)\lambda)\sum_{k\in\mathbb{Z}}d^k_{i,j}G_k,
\end{eqnarray*}
where $d^k_{i,j}\in\mathbb{C}$.

Secondly, take $x=L_i, y=L_j, u=L_i, v=L_j$ in Eq.\eqref{lem2}, we get
\begin{eqnarray*}
&&[(\phi_{\mu}(L_i, L_j))_{\mu+\gamma}[{L_i}_\lambda L_j]]=[[{L_i}_\mu L_j]_{\mu+\gamma}\phi_\lambda(L_i,L_j)]\\
\Longleftrightarrow&&[((\partial+2\mu)\sum_{k\in\mathbb{Z}}d^k_{i,j}L_k
+\sum_{k\in\mathbb{Z}}g^k_{i,j}(\partial,\mu)G_k)_{\mu+\gamma}(\partial+2\lambda)L_{i+j}]\\&&=
[((\partial+2\mu){L_{i+j}})_{\mu+\gamma}((\partial+2\lambda)\sum_{k\in\mathbb{Z}}d^k_{i,j}L_k
+\sum_{k\in\mathbb{Z}}g^k_{i,j}(\partial,\lambda)G_k)]\\
\Longleftrightarrow&&(\mu-\gamma)(\partial+\mu+\gamma+2\lambda)(\partial+2\mu+2\gamma)
\sum_{k\in\mathbb{Z}}d^k_{i,j}L_{i+j+k}\\&&
-(\partial+\mu+\gamma+2\lambda)(b\partial+(b-1)(\mu+\gamma))\sum_{k\in\mathbb{Z}}g^k_{i,j}(-\mu-\gamma,\mu)G_{i+j+k}
\\&&=(\mu-\gamma)(\partial+\mu+\gamma+2\lambda)(\partial+2\mu+2\gamma)
\sum_{k\in\mathbb{Z}}d^k_{i,j}L_{i+j+k}\\&&
+(\mu-\gamma)(\partial+(1-b)(\mu+\gamma))\sum_{k\in\mathbb{Z}}g^k_{i,j}(\partial+\mu+\gamma,\lambda)G_{i+j+k}\\
\Longleftrightarrow&&-(\partial+\mu+\gamma+2\lambda)(b\partial+(b-1)(\mu+\gamma))\sum_{k\in\mathbb{Z}}
g^k_{i,j}(-\mu-\gamma,\mu)G_{i+j+k},
\\&&=(\mu-\gamma)(\partial+(1-b)(\mu+\gamma))\sum_{k\in\mathbb{Z}}g^k_{i,j}(\partial+\mu+\gamma,\lambda)G_{i+j+k},
\end{eqnarray*}
it is easy to see that
\begin{eqnarray*}
&&-(\partial+\mu+\gamma+2\lambda)(b\partial+(b-1)(\mu+\gamma))g^k_{i,j}(-\mu-\gamma,\mu)
\\&&=(\mu-\gamma)(\partial+(1-b)(\mu+\gamma))g^k_{i,j}(\partial+\mu+\gamma,\lambda),
\end{eqnarray*}
for any $i,j,k\in\mathbb{Z}$. Considering the power of $\partial$, we have $g^k_{i,j}(\partial,\lambda)=c_0(\lambda)+c_1(\lambda)\partial$, substituting into the above equation, we get
\begin{eqnarray}
&&-(\partial+\mu+\gamma+2\lambda)(b\partial+(b-1)(\mu+\gamma))(c_0(\mu)+c_1(\mu)(-\mu-\gamma))\nonumber\\
&&=(\mu-\gamma)(\partial+(1-b)(\mu+\gamma))(c_0(\lambda)+c_1(\lambda)(\partial+\mu+\gamma))\label{lem3.1a}.
\end{eqnarray}
Considering the coefficients of $\gamma^3$, we have
\begin{eqnarray}
(b-1)(c_1(\mu)-c_1(\lambda))=0\label{lem3.1b}.
\end{eqnarray}
$(i)$ If $b=1$, substituting in Eq.\eqref{lem3.1a}, we get
\begin{eqnarray*}
-(\partial+\mu+\gamma+2\lambda)(c_0(\mu)+c_1(\mu)(-\mu-\gamma))
=(\mu-\gamma)(c_0(\lambda)+c_1(\lambda)(\partial+\mu+\gamma)),
\end{eqnarray*}
observing the coefficients of $\partial$ in the above equation, we have $c_0(\lambda)=c_1(\lambda)=0$.\\
$(ii)$ If $b\neq1$, by Eq.\eqref{lem3.1b}, we have $c_1(\lambda)=c_1\in\mathbb{C}$. Substituting in Eq.\eqref{lem3.1a}, we get
\begin{eqnarray*}
&&-(\partial+\mu+\gamma+2\lambda)(b\partial+(b-1)(\mu+\gamma))(c_0(\mu)+c_1(-\mu-\gamma))\nonumber\\
&&=(\mu-\gamma)(\partial+(1-b)(\mu+\gamma))(c_0(\lambda)+c_1(\partial+\mu+\gamma))\label{lem3.1},
\end{eqnarray*}
let's just think about the monomials only containing $\mu$, that is
\begin{eqnarray*}
(b-1)(c_0(\mu)-c_1\mu)=(b-1)c_1\mu,
\end{eqnarray*}
so we obtain $c_0(\mu)=2c_1\mu$. Substituting in Eq.\eqref{lem3.1a} again, we get
\begin{eqnarray*}
-c_1(b\partial+(b-1)(\mu+\gamma))=c_1(\partial+(1-b)(\mu+\gamma)),
\end{eqnarray*}
that is $c_1(b+1)=0$. Especially, if $b\neq-1$, then $c_1=0$.

From what has been discussed above, if\\
$(1)$ $b\neq-1$, we have
\begin{eqnarray*}
&&\phi_\lambda(L_i,L_j)=(\partial+2\lambda)\sum_{k\in\mathbb{Z}}d^k_{i,j}L_k,\\
&&\phi_\lambda(G_i,G_j)=0,\\
&&\phi_\lambda(L_i,G_j)=(\partial+(1-b)\lambda)\sum_{k\in\mathbb{Z}}d^k_{i,j}G_k;
\end{eqnarray*}
$(2)$ $b=-1$, we have
\begin{eqnarray*}
&&\phi_\lambda(L_i,L_j)=(\partial+2\lambda)(\sum_{k\in\mathbb{Z}}d^k_{i,j}L_k
+\sum_{k\in\mathbb{Z}}g^k_{i,j}G_k),\\
&&\phi_\lambda(G_i,G_j)=0,\\
&&\phi_\lambda(L_i,G_j)=(\partial+(1-b)\lambda)\sum_{k\in\mathbb{Z}}d^k_{i,j}G_k.
\end{eqnarray*}

Finally, take $x=L_i, y=L_j, u=L_k, v=L_l$ in Eq.\eqref{lem2}, we can get
$d^k_{i,j}=d^{k+m+n-i-j}_{m,n}$ and $g^k_{i,j}=g^{k+m+n-i-j}_{m,n}$. Especially, $d^k_{0,0}=d^{k+m+n}_{m,n}$ and $g^k_{0,0}=g^{k+m+n}_{m,n}$. Set $d^k_{0,0}:=a_k, g^k_{0,0}:=b_k$, i.e. $\phi_\lambda(L_0,L_0)=(\partial+2\lambda)(\sum_{k\in\mathbb{Z}}a_kL_k+\delta_{b+1,0}\sum_{k\in\mathbb{Z}}b_kG_k)$. Then
\begin{eqnarray}
&&\phi_\lambda(L_i,L_j)=(\partial+2\lambda)(\sum_{k\in\mathbb{Z}}a_{k-i-j}L_k
+\delta_{b+1,0}\sum_{k\in\mathbb{Z}}b_{k-i-j}G_k),\label{lem3.3a}\\
&&\phi_\lambda(G_i,G_j)=0,\label{lem3.3b}\\
&&\phi_\lambda(L_i,G_j)=(\partial+(1-b)\lambda)\sum_{k\in\mathbb{Z}}a_{k-i-j}G_k\label{lem3.3c}.
\end{eqnarray}

Conversely, it is easy to prove that any skew-symmetric conformal bilinear map satisfying Eqs.\eqref{lem3.3a}, \eqref{lem3.3b} and \eqref{lem3.3c} is a conformal biderivation.
\end{proof}


\begin{thebibliography}{99}
\bibitem{BKV} Bakalov, Bojko; Kac, Victor G.; Voronov, Alexander A. Cohomology of conformal algebras.
 \emph{Comm. Math. Phys.} \textbf{200} (1999), no. 3, 561-598.

\bibitem{B1} Bre$\check{s}$ar, Matej Centralizing mappings and derivations in prime rings. \emph{J. Algebra} \textbf{156} (1993), no. 2, 385-394.
\bibitem{B2} Bre$\check{s}$ar, Matej Commuting maps: a survey. \emph{Taiwanese J. Math.} \textbf{8} (2004), no. 3, 361-397.
\bibitem{BZ} Bre$\check{s}$ar, Matej; Zhao, Kaiming Biderivations and commuting linear maps on Lie algebras. \emph{J. Lie Theory} \textbf{28} (2018), no. 3, 885-900.
\bibitem{CSY}  Chi, Lili; Sun, Jiancai; Yang, Hengyun Biderivations and linear commuting maps on the planar Galilean conformal algebra. \emph{Linear Multilinear Algebra} \textbf{66} (2018), no. 8, 1606-1618.
\bibitem{DK} D'Andrea, Alessandro; Kac, Victor G. Structure theory of finite conformal algebras.
 \emph{Selecta Math. (N.S.)} \textbf{4} (1998), no. 3, 377-418.
\bibitem{FWY} Fan, Guangzhe; Wu, Henan; Yu, Bo Loop $W(a,b)$ Lie conformal algebra. \emph{Internat. J. Math.}
 \textbf{27}  (2016),  no. 2, 1650016, 13 pp.

\bibitem{HWX}   Han, Xiu; Wang, Dengyin; Xia, Chunguang Linear commuting maps and biderivations on the Lie algebras W(a,b). \emph{J. Lie Theory} \textbf{26} (2016), no. 3, 777-786.
\bibitem{LGZ} Liu, Xuewen; Guo, Xiangqian; Zhao, Kaiming Biderivations of the block Lie algebras.  \emph{Linear Algebra Appl.} \textbf{538} (2018), 43-55.
\bibitem{T1}   Tang, Xiaomin Biderivations and commutative post-Lie algebra structures on the Lie algebra W(a,b). \emph{Taiwanese J. Math.} \textbf{22} (2018), no. 6, 1347-1366.
\bibitem{T2}   Tang, Xiaomin Biderivations, linear commuting maps and commutative post-Lie algebra structures on W-algebras. \emph{Comm. Algebra} \textbf{45} (2017), no. 12, 5252-5261.


\bibitem{WY} Wang, Dengyin; Yu, Xiaoxiang Biderivations and linear commuting maps on the Schr$\ddot{o}$dinger-Virasoro Lie algebra. \emph{Comm. Algebra} \textbf{41} (2013), no. 6, 2166-2173.

 \bibitem{WCY} Wu, Henan; Chen, Qiufan; Yue, Xiaoqing Loop Virasoro Lie conformal algebra.
  \emph{J. Math. Phys.} \textbf{55} (2014), no. 1, 011706, 8 pp.
































\end{thebibliography}
\end{document}